\documentclass[12pt,pdftex,noinfoline]{imsart}

\RequirePackage[OT1]{fontenc}
\usepackage[linesnumbered]{algorithm2e}
\RequirePackage{amsthm,amsmath,amsfonts}
\RequirePackage{hypernat}
\usepackage{fullpage}
\usepackage{graphicx}
\usepackage{hyperref}
\usepackage{amsthm, amssymb, amsmath}
\usepackage{url}
\usepackage{tikz}
\usepackage[english]{babel}
\usepackage{times}
\usepackage[T1]{fontenc}
\usepackage{bbm}
\usepackage{color}
\usepackage{xspace}
\usepackage{rotating,multirow}
\usepackage{enumerate}
\usepackage{mathtools}
\usetikzlibrary{arrows,positioning}
\tikzset{
  tag/.style={
    rectangle,
    draw=black, thick,
    text width=4em,
    minimum height=3em,
    text centered},
}

\def\interleave{|\kern-.25ex|\kern-.25ex|}
\def\interleavesub{|\kern-.15ex|\kern-.15ex|}

\newcommand{\nNorm}[1]{\left|\kern-.25ex\left|\kern-.25ex\left| {#1}\right|\kern-.25ex\right|\kern-.25ex\right|}

\numberwithin{equation}{section}
\theoremstyle{plain}
\newtheorem{theorem}{Theorem}[section]
\newtheorem{corollary}{Corollary}[section]

\newtheorem{lemma}{Lemma}[section]
\newtheoremstyle{remark}{\topsep}{\topsep}%
     {\normalfont}
     {}           
     {\bfseries}  
     {.}          
     {.5em}       
     {\thmname{#1}\thmnumber{ #2}\thmnote{ #3}}
\theoremstyle{remark}

\long\def\comment#1{}
\def\reals{{\mathbb R}}
\def\P{{\mathbb P}}

\def\supp{\mathop{\text{supp}\kern.2ex}}
\def\argmin{\mathop{\text{\rm arg\,min}}}

\def\argmax{\mathop{\text{\rm arg\,max}}}
\let\hat\widehat
\let\tilde\widetilde

\let\hat\widehat
\let\tilde\widetilde
\def\given{{\,|\,}}

\def\1{{(1)}}
\def\2{{(2)}}

\def\M{{\mathcal{M}}}

\long\def\comment#1{}

\long\def\comment#1{}
\def\reals{{\mathbb R}}
\def\P{{\mathbb P}}

\def\supp{\mathop{\text{supp}\kern.2ex}}
\def\argmin{\mathop{\text{\rm arg\,min}}}

\def\argmax{\mathop{\text{\rm arg\,max}}}
\let\tilde\widetilde

\let\hat\widehat
\let\tilde\widetilde
\def\given{{\,|\,}}

\def\1{{(1)}}
\def\2{{(2)}}

\long\def\comment#1{}

\def\threebars{\mbox{$|\kern-.25ex|\kern-.25ex|$}}

\def\M{\mathbb{M}}
\def\F{\mathbb{F}}

\begin{document}

\begin{frontmatter}
\centerline{\large\bf Quantized Estimation of Gaussian Sequence Models
  in Euclidean Balls}
\runtitle{Quantized Estimation of Gaussian Sequences}

\begin{aug}
\vskip10pt
\author{\fnms{Yuancheng} \snm{Zhu}\ead[label=e1]{yuancheng@stanford.edu}}
\and
\author{\fnms{John} \snm{Lafferty}\ead[label=e4]{lafferty@galton.uchicago.edu}}
\address{
\vskip1pt
\begin{tabular}{c}
Department of Statistics \\
The University of Chicago 
\end{tabular}
\\[10pt]
\today\\[5pt]
}
\end{aug}

\begin{abstract}

  A central result in statistical theory is Pinsker's theorem, which
  characterizes the minimax rate in the normal means model of
  nonparametric estimation.  In this paper, we present an extension to
  Pinsker's theorem where estimation is carried out under storage
  or communication constraints.  In particular, we place limits on the number of bits
  used to encode an estimator, and analyze the excess risk in terms of
  this constraint, the signal size, and the noise level.  We give
  sharp upper and lower bounds for the case of a Euclidean ball, which
  establishes the Pareto-optimal minimax tradeoff between storage and
  risk in this setting.

\end{abstract}


\vskip20pt
\end{frontmatter}

\maketitle

\begin{abstract}

  A central result in statistical theory is Pinsker's theorem, which
  characterizes the minimax rate in the normal means model of
  nonparametric estimation.  In this paper, we present an extension to
  Pinsker's theorem where estimation is carried out under storage
  or communication constraints.  In particular, we place limits on the number of bits
  used to encode an estimator, and analyze the excess risk in terms of
  this constraint, the signal size, and the noise level.  We give
  sharp upper and lower bounds for the case of a Euclidean ball, which
  establishes the Pareto-optimal minimax tradeoff between storage and
  risk in this setting.

\end{abstract}


\section{Introduction}

Classical statistical theory studies the rate at which the error
in an estimation problem decreases as the sample size 
increases.    Methodology for a particular problem is developed to
make estimation efficient, and lower bounds establish
how quickly the error can decrease in principle.  Asymptotically
matching upper and lower bounds together yield the minimax
rate of convergence
$$ R_n(\F) = \inf_{\hat f} \sup_{f\in \F}  R(\hat f, f).$$
This is the worst-case error in estimating an element of a model class $\F$,
where $R(\hat f, f)$ is the risk or expected loss, and $\hat f$ 
is an estimator constructed on a data sample of size $n$.  The
corresponding sample complexity of the estimation problem is 
$n(\epsilon, \F) = \min\{n : R_n(\F) < \epsilon\}$.

In the classical setting, the infimum is over all
estimators.  In contemporary settings, it is increasingly of interest to understand how error depends
on computation.  For instance, when the data are high dimensional and the sample size
is large, constructing the estimator using standard
methods may be computationally prohibitive.   The use of
heuristics and approximation algorithms may make computation more efficient,
but it is important to understand the loss in statistical efficiency that
this incurs.   
In the minimax framework, this can be formulated by placing
computational constraints on the estimator:
$$ R_n(\F,B_n) = \inf_{\hat f: C(\hat f)\leq B_n} \sup_{f\in \F} R(\hat f, f).$$
Here $C(\hat f) \leq B_n$ indicates that the computation $C(\hat f)$
used to construct $\hat f$ is required to fall within a ``computational budget'' $B_n$.
Minimax lower bounds on the risk as a function of the computational
budget thus determine a feasible region for computation-constrained 
estimation, and a Pareto-optimal tradeoff for error versus computation.

One important measure of computation is the number of floating point
operations, or the running time of an algorithm.   Chandrasekaran and
Jordan \cite{chandrasekaran:13}
have studied upper bounds for
statistical estimation with computational constraints of this form
in the normal means model.  However, useful lower bounds are elusive.
This is due to 
the difficult nature of establishing tight lower bounds for this model
of computation in the polynomial hierarchy, apart from any statistical concerns.
%
%
Another important measure of computation is storage, or the space
used by a procedure. In particular, we may wish to limit the number of bits used to represent
our estimator $\hat f$.  The question then becomes, how 
does the excess risk depend on the budget $B_n$ imposed
on the number of bits $C(\hat f)$ used to 
encode the estimator?

This problem is naturally motivated by certain applications.
For instance, the Kepler telescope collects flux data for
approximately 150,000 stars \cite{2041-8205-713-2-L87}.  The central statistical task is
to estimate the lightcurve of each star nonparametrically, in order to denoise
and detect planet transits.   If this estimation is done 
on board the telescope, the estimated function values may need to be
sent back to earth for further analysis.  To limit communication
costs, the estimates can be quantized.  The fundamental question is,
what is lost in terms of statistical risk in quantizing the estimates?
Or, in a cloud computing environment (such as Amazon EC2), a large
number of nonparametric estimates might be constructed 
over a cluster of compute nodes and then stored (for example in
Amazon S3) for later analysis.  To limit the storage costs,
which could dominate the compute costs in many scenarios, it is of
interest to quantize the estimates.  How much is lost in terms of
risk, in principle, by using different levels of quantization?

With such applications as motivation, we address in this paper the
problem of risk-storage tradeoffs in the normal means model of
nonparametric estimation.  The normal means model is a centerpiece of
nonparametric estimation.  It arises naturally when representing an
estimator in terms of an orthogonal basis \cite{johnstone2002function,tsybakov:2008}.  Our
main result is a sharp characterization of the Pareto-optimal tradeoff
curve for quantized estimation of a normal means vector, in the
minimax sense.  We consider the case of a Euclidean ball of
unknown radius in $\reals^n$.  This case exhibits many of the key
technical challenges that arise in nonparametric estimation over
richer spaces, including the Stein phenomenon and the problem of
adaptivity.  

As will be apparent to the reader, the problem we consider is
intimately related to classical rate distortion theory \cite{gallager:1968}.
Indeed, our results require a marriage of minimax theory and rate
distortion ideas.  We thus build on the fundamental
connection between function estimation and lossy source coding that
was elucidated in Donoho's 1998 Wald Lectures \cite{donoho2000wald}.  This connection
can also be used to advantage for practical estimation schemes.  
As we discuss further below, recent advances on
computationally efficient, near-optimal lossy compression using sparse
regression algorithms  \cite{venkataramanan2013lossy} can perhaps be leveraged for quantized
nonparametric estimation.

In the following section, we present relevant background
and give a detailed statement of our results.  In Section~\ref{sec:main}
we sketch a proof of our main result
on the excess risk for the Euclidean ball case.  Section~\ref{sec:sims}
presents simulations to illustrate our theoretical
analyses. Section~\ref{sec:discuss} discusses related work, and outlines
future directions that our results suggest.

%
%
%
%

\def\M{{\mathcal M}}

\section{Background and problem formulation}
\label{sec:problem}

In this section we briefly review the essential elements of 
rate-distortion theory and minimax theory, to establish notation.
We then state our main result, which bridges
these classical theories.

\label{sec:ratedistortion}

In the rate-distortion setting we have a source that produces a sequence $X^n=(X_1,X_2,\dots, X_n)$,
each component of which is independent and identically distributed as $\mathcal N(0,\sigma^2)$. 
The goal is to transmit a realization from this sequence of random variables using a
fixed number of bits, in such a way that results in the
minimal expected distortion with respect to the original data $X^n$.
Suppose that we are allowed to use a total budget of $nB$ bits, so that the
average number of bits per variable is $B$, which is referred to as the \textit{rate}.
To transmit or store the data, the \emph{encoder} describes the source sequence $X^n$ by an index $\phi_n(X^n)$, where
\[
\phi_n:\mathbb R^n\to \{1,2,\dots,2^{nB}\} \equiv {\mathbb C}(B)
\]
is the \emph{encoding function}. The $nB$-bit index is then transmitted or stored without loss. 
A \emph{decoder}, when receiving or retrieving the data, represents $X^n$ by an estimate $\check X^n$ based on the index using a \emph{decoding function}
\[
\psi_n:\{1,2,\dots,2^{nB}\}\to\mathbb R^n.
\]
The image of the decoding function $\psi_n$ is called the \emph{codebook}, 
which is a discrete set in $\mathbb R^n$ with cardinality no larger than $2^{nB}$. 
The process is illustrated in Figure \ref{fig:datatrans}, and
variously referred to as source coding, lossy compression, or quantization.
We call the pair of encoding and decoding functions $Q_n=(\phi_n,\psi_n)$ an $(n,B)$-\emph{rate distortion code}. 
We will also use $Q_n$ to denote the composition of the two functions, i.e.,
$Q_n(\cdot)=\psi_n(\phi_n(\cdot))$.

\begin{figure}[t]
\vskip-5pt
\begin{center}
\tikzset{>=stealth',every on chain/.append style={join},
         every join/.style={->}}
\begin{tikzpicture}
\matrix[row sep=10mm,column sep=4mm ] {
&&&\node(source) {$X^n$}; 
&& \node [tag] (encoder) {Encoder $\phi_n$}; 
&&&&&&& \node[tag] (decoder) {Decoder $\psi_n$};
&& \node(output) {$\check X^n=\psi_n\left(\phi_n(X^n)\right)$}; \\
};
\begin{scope}[every node/.style={midway,auto,font=\scriptsize}]
\draw [shorten >=5pt,->] (source) -- (encoder);
\draw [shorten <=5pt, shorten >=5pt, ->] (encoder) -- node{$\phi_n(X^n)\in{\mathbb C}(B)$} (decoder);
\draw [shorten <=5pt, ->] (decoder) -- (output);
\end{scope}
\end{tikzpicture}
\end{center}
\begin{center}
\tikzset{>=stealth',every on chain/.append style={join},
         every join/.style={->}}
\begin{tikzpicture}
\hskip-5pt\matrix[row sep=10mm,column sep=4mm ] {
\node(model) {$\theta^n$}; 
&& \node(source) {$X^n$}; 
&& \node [tag] (encoder) {Encoder $\phi_n$}; 
&&&&&&& \node[tag] (decoder) {Decoder $\psi_n$};
&& \node(output) {$\check \theta^n=\psi_n\left(\phi_n(X^n)\right)$}; \\
};
\begin{scope}[every node/.style={midway,auto,font=\scriptsize}]
\draw [shorten >=5pt,->] (model) -- (source);
\draw [shorten >=5pt,->] (source) -- (encoder);
\draw [shorten <=5pt, shorten >=5pt, ->] (encoder) -- node{$\phi_n(X^n)\in{\mathbb C}(B)$} (decoder);
\draw [shorten <=5pt, ->] (decoder) -- (output);
\end{scope}
\end{tikzpicture}
\end{center}
\vskip-5pt
\caption{Encoding and decoding process for lossy compression (top) and
  quantized estimation (bottom).  For quantized estimation, the
  model (mean vector) $\theta^n$ is deterministic, not
  random.}\label{fig:datatrans}
\vskip-5pt
\end{figure}
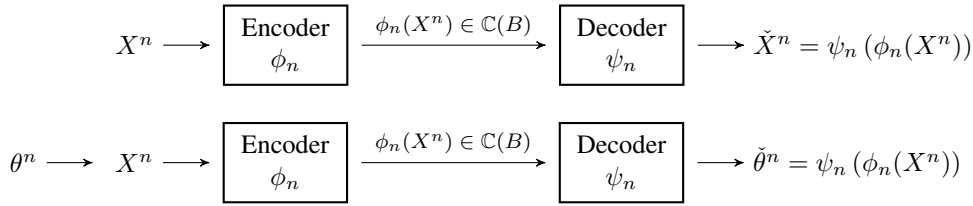


A \emph{distortion measure}, or a \emph{loss function}, $d:\mathbb R\times\mathbb R\to\mathbb R^+$ 
is used to evaluate the performance of the above coding and transmission process. 
In this paper, we will use the squared loss $d(X_i,\check X_i)=(X_i-\check X_i)^2$. 
The distortion between two sequences $X^n$ and $\check X^n$ is then defined by $d_n(X^n,\check X^n)=\frac{1}{n}\sum_{i=1}^n(X_i-\check X_i)^2$, the average of the per observation distortions. We drop the subscript $n$ in $d$ when it is clear from the context. The \emph{distortion}, or \emph{risk}, for a $(n,B)$-rate distortion code $Q_n$ is defined as the expected loss $\mathbb E\,  d\left(X^n,Q_n(X^n)\right)$. Denoting by $\mathcal Q_{n,B}$ the set of all $(n,B)$-rate distortion codes, the \emph{distortion rate function} is defined as
\[
R(B,\sigma)=\liminf_{n\to\infty}\inf_{Q_n\in\mathcal Q_{n,B}}\mathbb E\,  d\left(X^n,Q_n(X^n)\right).
\]
This distortion rate function depends on the rate $B$ as well as the
source distribution. For the i.i.d. $\mathcal N(0,\sigma^2)$ source,
according to the well-known rate distortion theorem \cite{gallager:1968},
\[
R(B,\sigma)=\sigma^22^{-2B}.
\]
When $B$ is zero, meaning no information gets encoded at all, this
bound becomes $\sigma^2$, which is the expected loss when each random
variable is represented by its mean. As $B$ approaches infinity, the
distortion goes to zero.


The previous discussion assumes the source random variables are
independent and follow a common distribution $\mathcal
N(0,\sigma^2)$. The goal is to minimize the expected distortion in the
reconstruction of $X^n$ after
transmitting or storing the data under a communication constraint. Now
suppose that
\[
X_i\overset{\text{ind.}}{\sim}\mathcal N(\theta_i,\sigma^2)\ \text{ for $i=1,2,\dots,n$}.
\]
We assume the variance $\sigma^2$ is known and the means $\theta^n=(\theta_1,\dots,\theta_n)$ are unknown.
Suppose, furthermore, 
that instead of trying to minimize the recovery distortion $d(X^n,\check X^n)$, 
we want to estimate the means with a risk as small as possible, but
again using a budget of $B$ bits per index.

Without the communication constraint, this problem has been very well
studied \cite{pinsker1980optimal,nussbaum1999minimax}. Let
$\hat\theta(X^n) \equiv \hat\theta^n= (\hat\theta_1,\dots,\hat\theta_n)$ denote an estimator
of the true mean $\theta^n$. For a parameter space
$\Theta_n\subset\mathbb R^n$, the minimax risk over $\Theta_n$ is
defined as
\[
\inf_{\hat\theta^n}\sup_{\theta^n\in\Theta_n}\mathbb E\,d(\theta^n,\hat\theta^n)=\inf_{\hat\theta^n}\sup_{\theta^n\in\Theta_n}\mathbb E\,\frac{1}{n}\sum_{i=1}^n(\theta_i-\hat\theta_i)^2.
\]
For the $L_2$ ball of radius $c$,
\begin{equation}
\Theta_n(c)=\Bigl\{(\theta_1,\dots,\theta_n):\frac{1}{n}\sum_{i=1}^n\theta_i^2\leq c^2\Bigr\},\label{l2ball}
\end{equation}
Pinsker's theorem gives the exact, limiting form of the minimax risk
\[
\liminf_{n\to\infty}\inf_{\hat\theta^n}\sup_{\theta^n\in\Theta_n(c)}\mathbb E\,d(\theta^n,\hat\theta^n)=\frac{\sigma^2c^2}{\sigma^2+c^2}.
\]

To impose a communication constraint, we incorporate a variant of the source coding
scheme described above into this minimax
framework of estimation. Define a $(n,B)$-rate estimation code
$M_n=(\phi_n,\psi_n)$, as a pair of encoding and decoding functions,
as before.  The encoding function $\phi_n: \reals^n \rightarrow \{1,2,\ldots,
2^{nB}\}$ is a mapping from observations $X^n$ to an index set.
The decoding function is a mapping from indices to models $\check \theta^n \in\reals^n$.
We write the composition of the encoder and decoder as
$M_n(X^n) = \psi_n(\phi_n(X^n)) = \check\theta^n$,
which we call a \textit{quantized estimator}.
Denoting by $\mathcal M_{n,B}$ the set of all $(n,B)$-rate estimation codes,
we then define the \textit{quantized minimax risk} as
\[
R_n(B,\sigma,\Theta_n)=\inf_{M_n\in\mathcal M_{n,B}}\sup_{\theta^n\in\Theta_n}\mathbb E\, d(\theta^n,M_n(X^n)).
\]
We will focus on the case where our parameter space is the $L_2$ ball
defined in \eqref{l2ball}, and write 
\[
R_n(B,\sigma,c)=R_n(B,\sigma,\Theta_n(c)).
\]
In this setting, we let $n$ go to infinity and define the asymptotic
quantized minimax risk as
\begin{equation}
R(B,\sigma,c)=\liminf_{n\to\infty}R_n(B,\sigma,c)=\liminf_{n\to\infty}\inf_{M_n\in\mathcal M_{n,B}}\sup_{\theta^n\in\Theta_n(c)}\mathbb E\, d(\theta^n,M_n(X^n)).\label{aympminimax}
\end{equation}
Note that we could estimate $\theta^n$ 
based on the quantized data $\check X^n=Q_n(X^n)$. Once again denoting
by $\mathcal Q_{n,B}$ the set of all $(n,B)$-rate distortion codes,
such an estimator is written
$\check\theta^n=\check\theta^n(Q_n(X^n))$. 
Clearly, if the decoding functions $\psi_n$ of $\mathcal Q_n$ are
injective, then this formulation is equivalent.  The quantized minimax risk
is then expressed as
\[
R_n(B,\sigma,\Theta_n)=\inf_{\check\theta^n}\inf_{Q_n\in\mathcal Q_{n,B}}\sup_{\theta^n\in\Theta_n}\mathbb E\, d(\theta^n,\check\theta^n).
\]

The many normal means problem exhibits much of the complexity and
subtlety of general nonparametric regression and density
estimation problems.  It arises naturally in the estimation of a function
expressed in terms of an orthogonal function basis
\cite{johnstone2002function,wasserman:2006}.    
Our main result
sharply characterizes the excess risk that communication constraints
impose on minimax estimation for $\Theta(c)$.


\begin{figure}[t]
\begin{center}
\ \vskip-10pt
\begin{tabular}{ll}
\hskip-10pt
\includegraphics[width=.4\textwidth]{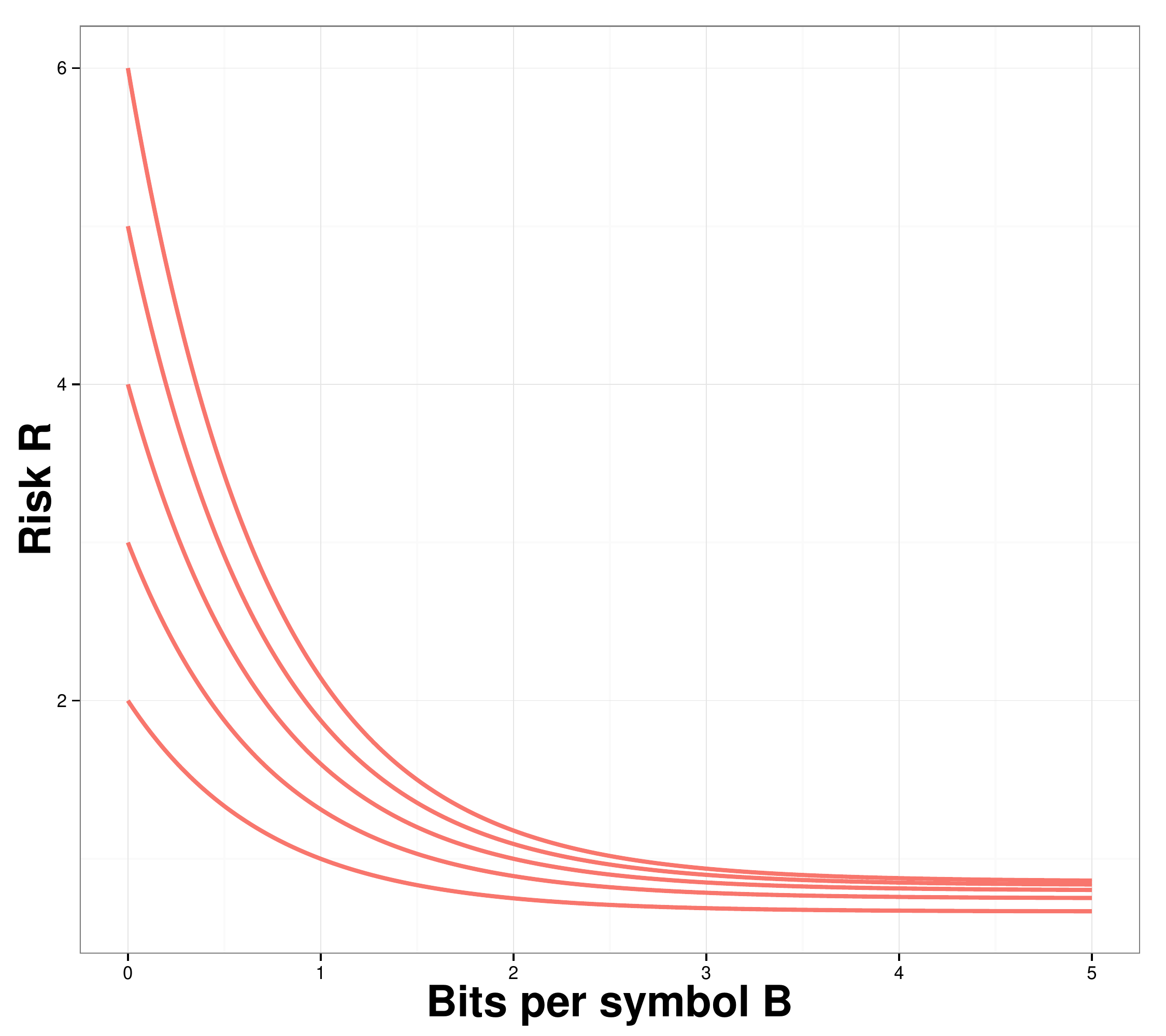} \quad & 
\quad \begin{minipage}[b]{2.5in}
\small\linespread{1.0}\selectfont
\stepcounter{figure}
Figure \arabic{figure}. Our result establishes the Pareto-optimal tradeoff 
in the nonparametric normal means problem for 
risk versus number of bits:
$$ R(\sigma^2, c^2, B) = \frac{c^2\sigma^2}{\sigma^2 + c^2} + \frac{c^4 2^{-2B}}{\sigma^2+c^2}$$
\!\!\!Curves for five signal sizes are shown, $c^2 = 2, 3, 4, 5,
6$.  The noise level is $\sigma^2=1$.  With zero bits, the
rate is $c^2$, the highest point on the risk curve.  The rate
for large $B$ approaches the Pinsker bound $\sigma^2 c^2 / (\sigma^2
+ c^2)$.
\vskip2pt{\ }
\end{minipage}
\label{tradeoff}
\\[-10pt]
\end{tabular}
\end{center}
\end{figure}

\section{Main results}
\label{sec:main}

Our first result gives a lower bound on the exact quantized asymptotic risk in terms
of $B$, $\sigma$, and $c$. 

\begin{theorem}\label{thm:lowerbound} For $B\geq 0$, $\sigma>0$ and $c>0$, the asymptotic minimax risk defined in (\ref{aympminimax}) satisfies
\begin{equation}
R(B,\sigma,c)\geq\frac{\sigma^2c^2}{\sigma^2+c^2}+\frac{c^4}{\sigma^2+c^2}2^{-2B}.\label{minimax}
\end{equation}
\end{theorem}
This lower bound on the limiting minimax risk can be viewed as the
usual minimax risk without quantization, plus an excess risk term due to
quantization. If we take $B$ to be zero, the risk becomes $c^2$, which is
obtained by estimating all of the means simply by zero.
On the other hand, letting $B\to\infty$, we recover the
minimax risk in Pinsker's theorem.  This tradeoff is
illustrated in Figure~\ref{tradeoff}.

The proof of the theorem is technical and we defer it to the
supplementary material. Here we sketch the basic idea of the proof. Suppose
we are able to find a prior distribution $\pi_n$ on $\theta^n$ and a
random vector $\tilde\theta^n$ such that for any $(n,B)$-rate
estimation code $M_n$ the following holds:
\begin{align*}
\frac{\sigma^2c^2}{\sigma^2+c^2}+\frac{c^4}{\sigma^2+c^2}2^{-2B}&\overset{(\text{I})}{=}\int \mathbb E_{X^n}d(\theta^n,\tilde\theta^n)d\pi_n(\theta^n)\\
&\overset{(\text{II})}{\leq} \int \mathbb E_{X^n}d(\theta^n,M_n(X^n))d\pi_n(\theta^n)\\
&\overset{(\text{III})}{\leq} \sup_{\theta^n\in\Theta_n(c)}\mathbb E_{X^n}d(\theta^n,M_n(X^n)).
\end{align*}
Then taking an infimum over $M_n\in\mathcal M_{n,B}$ gives us the desired result. In fact, we can take $\pi_n$, the prior on $\theta^n$, to be $\mathcal N(0,c^2I_n)$, and the model becomes $\theta_i\sim\mathcal N(0,c^2)$ and $X_i\given \theta_i\sim\mathcal N(\theta_i,\sigma^2)$. Then according to Lemma \ref{lemmaratedistortion}, inequality (II) holds with $\tilde\theta^n$ being the minimizer to the optimization problem
\begin{align*}
\min_{p(\tilde\theta^n\given X^n,\theta^n)}\quad & \mathbb E\,d(\theta^n,\tilde\theta^n) \\
\text{subject to}\quad & I(X^n;\tilde\theta^n)\leq nB,\\
& p(\tilde\theta^n\given X^n,\theta^n)=p(\tilde\theta^n\given X^n).
\end{align*}
The equality (I) holds due to Lemma \ref{lem:normalcase}. The
inequality (III) can be shown by a limiting concentration argument on
the prior distribution, which is included in the supplementary
material.
\begin{lemma}\label{lemmaratedistortion}
Suppose that $X_1,\dots,X_n$ are independent and generated by
$\theta_i\sim \pi(\theta_i)$ and $X_i\given \theta_i\sim p(x_i \given \theta_i)$. Suppose $M_n$ is an $(n,B)$-rate estimation code with risk $\mathbb E\,d(\theta^n,M_n(X^n))\leq D$. Then the rate $B$ is lower bounded by the solution to the following problem:
\begin{align}
\min_{p(\tilde\theta^n\given X^n,\theta^n)}\quad &
I(X^n;\tilde\theta^n) \nonumber\\ 
\text{subject to}\quad \ & \mathbb E\, d(\theta^n,\tilde\theta^n)\leq D,\label{optimization}\\
& p(\tilde\theta^n\given X^n,\theta^n)=p(\tilde\theta^n\given X^n). \nonumber
\end{align}
\end{lemma}

The next lemma gives the solution to problem (\ref{optimization}) when we have $\theta_i\sim\mathcal N(0,c^2)$ and $X_i\given \theta_i\sim\mathcal N(\theta_i,\sigma^2)$

\begin{lemma}\label{lem:normalcase}
Suppose $\theta_i\sim\mathcal N(0,c^2)$ and $X_i\given \theta_i\sim\mathcal N(\theta_i,\sigma^2)$ for $i=1,\dots,n$. For any random vector $\tilde\theta^n$ satisfying $\mathbb E\, d(\theta^n,\tilde\theta^n)\leq D$ and $p(\tilde\theta^n\given X^n,\theta^n)=p(\tilde\theta^n\given X^n)$ we have
\[
I(X^n;\tilde\theta^n) \geq \frac{n}{2}\log\frac{c^4}{(\sigma^2+c^2)(D-\frac{\sigma^2c^2}{\sigma^2+c^2})}.
\]
\end{lemma}
Combining the above two lemmas, we obtain a lower bound of the risk assuming that $\theta^n$ follows the prior distribution $\pi_n$:
\begin{corollary}\label{cor:lowerbound}
Suppose $M_n$ is a $(n,B)$-rate estimation code for the source $\theta_i\sim\mathcal N(0,c^2)$ and $X_i\given \theta_i\sim\mathcal N(\theta_i,\sigma^2)$, then
\begin{equation}
\mathbb E\,d(\theta^n,M_n(X^n))\geq \frac{\sigma^2c^2}{\sigma^2+c^2}+\frac{c^4}{\sigma^2+c^2}2^{-2B}.
\end{equation} 
\end{corollary}

\subsection{An adaptive source coding method}

We now present a source coding method, which we will show attains
the minimax lower bound asymptotically with high probability.

Suppose that the encoder is given a sequence of observations
$(X_1,\dots,X_n)$, and both the encoder and the decoder know the radius
$c$ of the $L_2$ ball in which the mean vector lies. The steps of the
source coding method are outlined below:

\begin{enumerate}[Step 1.]
\item Generating codebooks. The codebooks are distributed to both the encoder and the decoder.
\begin{enumerate}
\item Generate codebook $\mathcal B=\{1/\sqrt{n},2/\sqrt{n},\dots,\lceil c^2\sqrt{n}\rceil/\sqrt{n}\}$.
\item Generate codebook $\mathcal X$ which consists of $2^{nB}$ i.i.d.\ random vectors from the uniform distribution on the $n$-dimensional unit sphere $\mathbb S^{n-1}$.
\end{enumerate}
\item Encoding.
\begin{enumerate}
\item Encode $\hat b^2=\frac{1}{n}\|X\|^2-\sigma^2$ by
\[
\varphi_\mathcal B(\hat b^2)=\argmin\{|b^2-\hat b^2|:b^2\in\mathcal B\}.
\]
\item Encode $X^n$ by
\[
\varphi_{\mathcal X}(X^n)=\argmax\{\langle X^n,x^n\rangle:x^n\in\mathcal X\}
\]
\end{enumerate}
\item Transmit or store $\left(\varphi_\mathcal B(\hat b^2),\varphi_{\mathcal X}(X^n)\right)$  using $\log c^2+\frac{1}{2}\log n+nB$ bits.
\item Decoding.
\begin{enumerate}
\item Decode $\check b^2=\psi_\mathcal B\left(\varphi_\mathcal B(\hat b^2)\right)$ and $\check X^n=\psi_\mathcal X\left(\varphi_\mathcal X(X^n)\right)$, where the decoding function $\psi_\mathcal C(i)$ returns the $i$th element in the codebook $\mathcal C$.
\item Estimate $\theta$ by 
\[
\check\theta^n=\sqrt{\frac{n\check b^4(1-2^{-2B})}{\check b^2+\sigma^2}}\cdot \check X^n.
\]
\end{enumerate}
\end{enumerate}
We make several remarks on this quantized estimation method.

\textbf{Remark 1.} The rate of this coding method is $B+\frac{\log
  c^2}{n}+\frac{\log n}{2n}$, which is asymptotically $B$ bits.

\textbf{Remark 2.} The method is probabilistic; the randomness comes
from the construction of the codebook $\mathcal X$. Denoting by
$\mathcal M^*_{n,B,\sigma,c}$ the ensemble of such random quantizers,
there is then a natural one-to-one mapping between
$\mathcal M^*_{n,B,\sigma,c}$ and $(\mathbb S^{n-1})^{2^{nB}}$ and we
attach probability measure to $\mathcal M^*_{n,B,\sigma,c}$
corresponding to the product uniform distribution on $(\mathbb
S^{n-1})^{2^{nB}}$.

\textbf{Remark 3.} The main idea behind this coding scheme is to
encode the magnitude and the direction of the observation vector
separately, in such a way that the procedure adapts to sources with
different norms of the mean vectors.

\textbf{Remark 4.} The computational complexity of this source coding
method is exponential in $n$. Therefore, like the Shannon random
codebook, this is a demonstration of the asymptotic
achievability of the lower bound (\ref{minimax}), rather than a practical
scheme to be implemented. We discuss possible computationally
efficient algorithms in Section \ref{sec:discuss}.

The following shows that with high probability this procedure will
attain the desired lower bound asymptotically.

\begin{theorem}\label{thm:codingmethod} For a sequence of vectors $\{\theta^n\}_{n=1}^\infty$ satisfying $\theta^n\in\mathbb R^n$ and $\|\theta^n\|^2/n=b^2\leq c^2$, as $n\to\infty$
\begin{equation}
\mathbb
P\left(d(\theta^n,M_n(X^n))>\frac{\sigma^2b^2}{\sigma^2+b^2}+\frac{b^4}{\sigma^2+b^2}2^{-2B}+C\sqrt{\frac{\log
      n}{n}}\right)\longrightarrow 0
\end{equation}
for some constant $C$ that does not depend on $n$ (but could possibly depend on $b$, $\sigma$ and $B$). The probability measure is with respect to both $M_n\in\mathcal M^*_{n,B,\sigma,c}$ and $X^n\in\mathbb R^n$.
\end{theorem}

This theorem shows that the source coding method not only achieves the desired minimax lower bound for the $L_2$ ball with high probability with respect to the random codebook and source distribution, but also adapts to the true magnitude of the mean vector $\theta^n$. It agrees with the intuition that the hardest mean vector to estimate lies on the boundary of the $L_2$ ball. Based on Theorem \ref{thm:codingmethod} we can obtain a uniform high probability bound for mean vectors in the $L_2$ ball.

\begin{corollary}
For any sequence of vectors $\{\theta^n\}_{n=1}^\infty$ satisfying $\theta^n\in\mathbb R^n$ and $\|\theta^n\|^2/n\leq c^2$, as $n\to\infty$
\begin{equation*}
\mathbb
P\left(d(\theta^n,M_n(X^n))>\frac{\sigma^2c^2}{\sigma^2+c^2}+\frac{c^4}{\sigma^2+c^2}2^{-2B}+C'\sqrt{\frac{\log
      n}{n}}\right)
\longrightarrow 0
\end{equation*}
for some constant $C'$ that does not depend on $n$.
\end{corollary}

We include the details of the proof of Theorem \ref{thm:codingmethod}
in the supplementary material, which carefully analyzes the three
terms in the following decomposition of the loss function:
\begin{align*}
d(\theta^n,\check\theta^n)&=\frac{1}{n}\left\|\check\theta^n-\theta^n\right\|^2\\
&=\frac{1}{n}\left\|\check\theta^n-\hat\gamma X^n+\hat\gamma X^n-\theta^n\right\|^2\\
&=\underbrace{\frac{1}{n}\left\|\check\theta^n-\hat\gamma X^n\right\|^2}_{A_1}+\underbrace{\frac{1}{n}\left\|\hat\gamma X^n-\theta^n\right\|^2}_{A_2}+\underbrace{\frac{2}{n}\langle\check\theta^n-\hat \gamma X^n,\hat\gamma X^n-\theta^n\rangle}_{A_3}
\end{align*}
where $\hat\gamma=\frac{\hat
  b^2}{\hat b^2+\sigma^2}$ with $\hat b^2=\|X^n\|^2/n-\sigma^2$. Term $A_1$ characterizes the quantization
error. Term $A_2$ does not involve the random codebook, and is the loss of
a type of James-Stein estimator. The cross term $A_3$ vanishes as $n\to\infty$.


\section{Simulations}
\label{sec:sims}

In this section we present a set of simulation results showing the
empirical performance of the proposed quantized estimation method. Throughout
the simulation, we fix the noise level $\sigma^2=1$, while varying the
other parameters $c$ and $B$.

First we show in Figure \ref{fig:shrinkage} the effect of quantized
estimation and compare it with the James-Stein estimator.
Setting $n=15$ and
$c=2$, we randomly generate a mean vector $\theta^n\in\mathbb R^{n}$
with $\|\theta\|^2/n=c^2$. A random vector $X$ is then drawn from
$\mathcal N(\theta^n,I_n)$ and quantized estimates with rates
$B\in\{0.1,0.2,0.5,1\}$ are calculated; for comparison we also compute
the James-Stein estimator, given by $ \hat\theta_{\mbox{\tiny JS}}^n =
\left(1-\frac{(n-2)\sigma^2}{\|X^n\|^2}\right) X^n.$ 
We repeat this sampling and estimation procedure 100
times and report the averaged risk estimates in Figure
\ref{fig:shrinkage}. We see that the quantized estimator essentially
shrinks the random vector towards zero. With small rates, the shrinkage
is strong, with all the estimates close to zero.  Estimates with larger
rates approach the James-Stein estimator.
\begin{figure}[t]
\centering
\hskip20pt\includegraphics[width=5.2in]{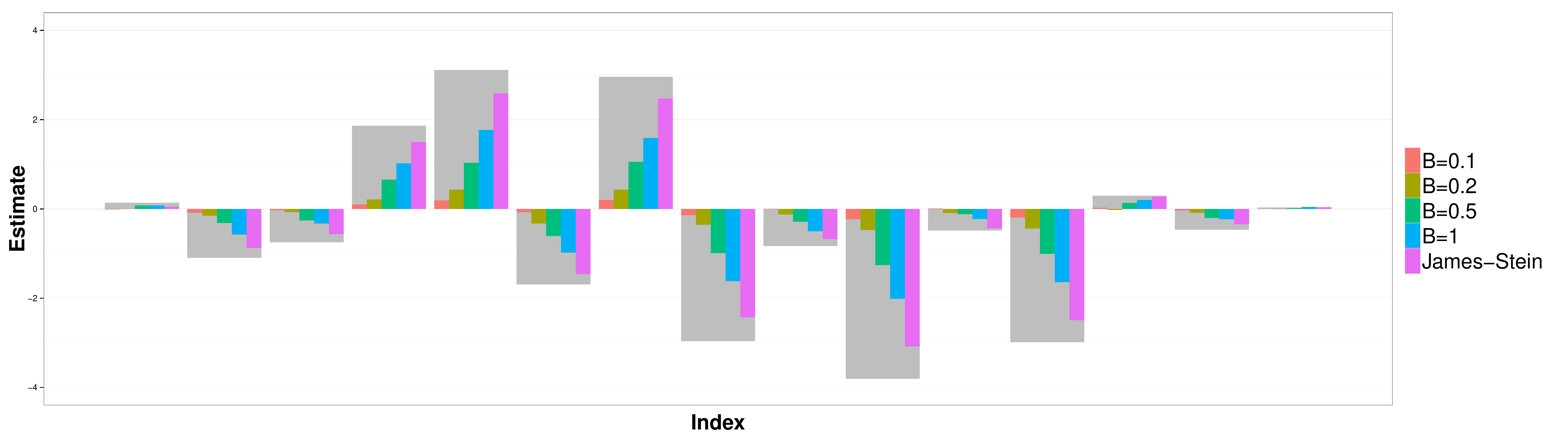}
\caption{\small Comparison of the quantized estimates with different
  rates $B$, the James-Stein estimator, and the true mean vector. The
  heights of the bars are the averaged estimates based on 100
  replicates. Each large background rectangle indicates the original
  mean component $\theta_j$.}
\label{fig:shrinkage}
\end{figure}

In our second set of simulations, we choose $c$ from
$\{0.1,0.5,1,5,10\}$ to reflect different signal-to-noise ratios, and
choose $B$ from $\{0.1,0.2,0.5,1\}$. For each combination of the
values of $c$ and $B$, we vary $n$, the dimension of the mean vector,
which is also the number of observations. Given a set of parameters $c$, $B$ and $n$, a mean
vector $\theta^n$ is generated uniformly on the sphere
$\|\theta^n\|^2/n=c^2$ and data $X^n$ are generated following the
distribution $\mathcal N(\theta^n,\sigma^2I_n)$. We quantize the
data using the source coding method, and compute the mean squared
error between the estimator and the true mean vector. The procedure is
repeated 100 times for each of the parameter combinations, and the
average and standard deviation of the mean squared errors are
recorded. The results are shown in Figure \ref{fig:simulation}. We see
that as $n$ increases, the average error decreases and approaches the
theoretic lower bound in Theorem \ref{thm:lowerbound}. Moreover, the
standard deviation of the mean squared errors also decreases, 
confirming the result of Theorem \ref{thm:codingmethod} that the
convergence is with high probability.
\begin{figure}[t]
\centering
\includegraphics[width=5.5in]{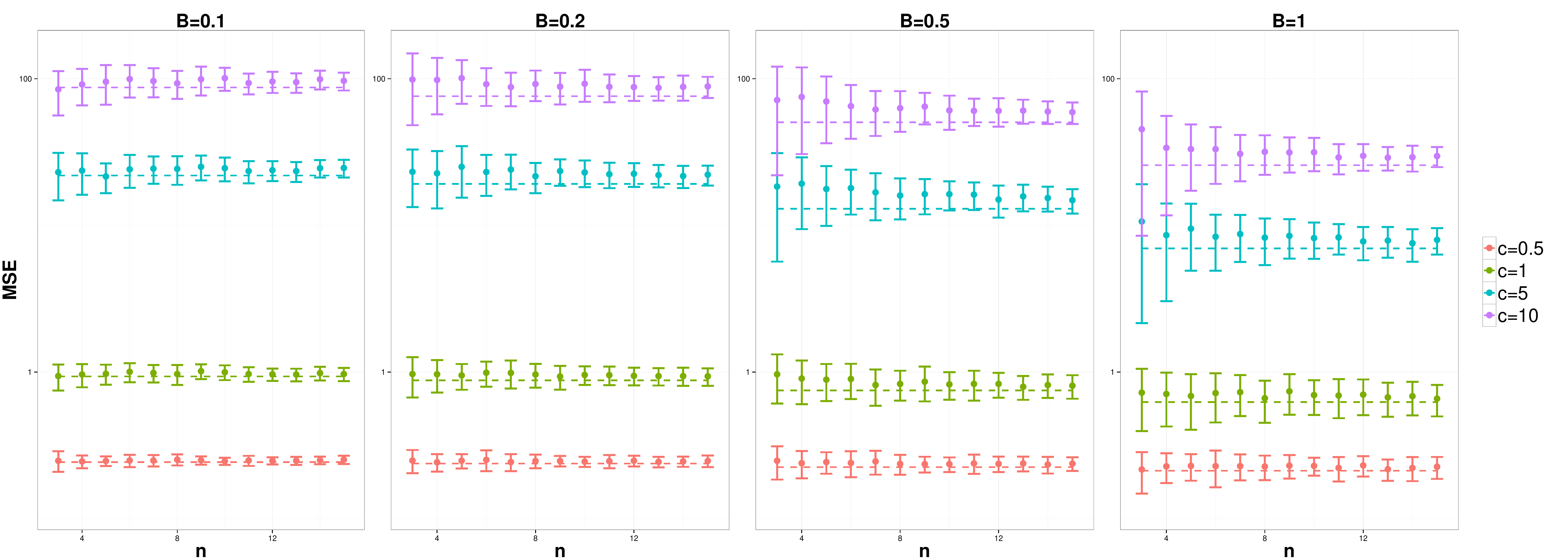}
\caption{Mean squared errors and standard deviations of the quantized estimator versus $n$ for different values of $(B,c)$. The horizontal dashed lines indicate the lower bounds.}
\label{fig:simulation}
\end{figure}

\section{Discussion and future work}
\label{sec:discuss}
\label{sec:related}

In this paper, we establish a sharp lower bound on the asymptotic
minimax risk for quantized estimators of nonparametric normal means
for the case of a Euclidean ball. Similar techniques can be applied to
the setting where the parameter space is an ellipsoid
$\Theta=\left\{\theta:\sum_{j=1}^\infty a_j^2\theta_j^2\leq
c^2\right\}$.  A principal case of interest is the Sobolev ellipsoid
of order $m$ where $a_j^2\sim (\pi j)^{2m}$ as $j\to\infty$. The Sobolev ellipsoid arises
naturally in nonparametric function estimation and is thus of great
importance. We leave this to future work.

Donoho discusses the parallel between rate distortion theory and
Pinsker's work in his Wald Lectures~\cite{donoho2000wald}. Focusing on
the case of the Sobolev space of order $m$, which we denote by $\F_m$,
it is shown that the Kolmogorov entropy $H_\epsilon(\F_m)$ and the
rate distortion function $R(D,X)$ satisfy $ H_\epsilon(\F_m) \asymp
\sup\{R(\epsilon^2, X) : {\mathbb P}(X\in \F_m)=1\}$ as
$\epsilon\rightarrow 0$.  This connects the worst-case minimax
analysis and least-favorable rate distortion function for the function
class.  Another information-theoretic formulation of minimax rates
lies in the so-called ``le Cam equation'' $H_\epsilon(\F) = n
\epsilon^2$ \cite{Wong:Shen:1995,Yang:Barron:1999}.  However, both are
different from the direction we pursue in this paper, which is to
impose communication constraints in minimax analysis.

In other related work, researchers in communications theory have
studied estimation problems in sensor networks under communication
constraints. Draper and Wornell \cite{draper2004side} obtain a result
on the so-called ``one-step problem'' for the quadratic-Gaussian case,
which is essentially the same as the statement in our Corollary
\ref{cor:lowerbound}. In fact, they consider a similar setting, but
treat the mean vector as random and generated independently from a
known normal distribution.  In contrast, we assume a fixed but unknown
mean vector and establish a minimax lower bound as well as an adaptive
source coding method that adapts to the fixed mean vector within the
parameter space.  Zhang et al.~\cite{zhang2013information} also
consider minimax bounds with communication constraints.  However, the
analysis in \cite{zhang2013information} is focused on distributed
parametric estimation, where the data are distributed between several
machines.  Information is shared between the machines in order to
construct a parameter estimate, and constraints are placed on the
amount of communication that is allowed.

In addition to treating more general ellipsoids, an important
direction for future work is to design computationally efficient
quantized nonparametric estimators.  One possible method is to divide
the variables into smaller blocks and quantize them separately.  A
more interesting and promising approach is to adapt the recent work of
Venkataramanan et al.~\cite{venkataramanan2013lossy} that uses sparse
regression for lossy compression.  We anticipate that with appropriate
modifications, this scheme can be applied to quantized nonparametric
estimation to yield practical algorithms, trading off a worse error
exponent in the convergence rate to the optimal quantized minimax risk
for reduced complexity encoders and decoders.





\appendix
\section{Proofs}
\subsection{Proof of Theorem \ref{thm:lowerbound}}

\begin{proof}[Proof of Lemma \ref{lemmaratedistortion}]
Denote the solution to problem (\ref{optimization}) by $B^*(D)$. Suppose that $M_n$ is an $(n,B)$-rate estimation code $M_n$ with risk $\mathbb E\,d(\theta^n,M_n(X^n))\leq D$. We have
\begin{align}
B&\geq I(X^n;M_n(X^n))/n\label{eqn:lem1.1}\\
&\geq B^*\left(\mathbb E\,d(\theta^n,M_n(X^n))\right)\label{eqn:lem1.2}\\
&\geq B^*(D)\label{eqn:lem1.3},
\end{align}
where (\ref{eqn:lem1.1}) follows from the fact that $M_n(X^n)$ takes at most $2^{nB}$ values; (\ref{eqn:lem1.2}) follows from the definition of $B^*(\cdot)$; (\ref{eqn:lem1.3}) follows from the monotonicity of $B^*(\cdot)$ and the fact that $\mathbb E\,d(\theta^n,M_n(X^n))\leq D$.
\end{proof}

\begin{proof}[Proof of Lemma \ref{lem:normalcase}]
Suppose that $\tilde\theta^n$ satisfies the conditions in problem (\ref{optimization}). Write $\gamma=c^2/(\sigma^2+c^2)$. For $i=1,2,\dots,n$, consider the decomposition
\begin{align*}
\mathbb E (\theta_i-\tilde\theta_i)^2&=\mathbb E (\theta_i-\gamma X_i+\gamma X_i-\tilde\theta_i)^2\\
&=\mathbb E(\theta_i-\gamma X_i)^2+\mathbb E(\tilde\theta_i-\gamma X_i)^2-2\mathbb E\left((\theta_i-\gamma X_i)(\tilde\theta_i-\gamma X_i)\right)\\
&=\frac{\sigma^2c^2}{\sigma^2+c^2}+\mathbb E(\tilde\theta_i-\gamma X_i)^2.
\end{align*}
The last equality follows from
\[
\mathbb E\left((\tilde\theta_i-\gamma X_i)(\theta_i-\gamma X_i)\right)=\mathbb E\left(\mathbb E(\theta_i-\gamma X_i\given X_i)\mathbb E(\tilde\theta_i-\gamma X_i\given X_i)\right)=0,
\]
where we have used the fact that $\theta_i\to X_i\to\tilde\theta_i$ is a Markov chain and that $\mathbb E(\theta_i\given X_i)=\gamma X_i$. Summing over $i=1,\dots,n$, we have
\[
\mathbb E\,d(\theta^n,\tilde\theta^n)=\frac{\sigma^2c^2}{\sigma^2+c^2}+\mathbb E\,d(\tilde\theta^n,\gamma X^n).
\]
A lower bound on the mutual information can be obtained as
\begin{align}
\frac{1}{n}I(X^n;\tilde\theta^n)&=\frac{1}{n}I(\gamma X^n;\tilde\theta^n)\geq\frac{1}{n}\sum_{i=1}^nI(\gamma X_i;\tilde\theta_i)\nonumber\\
&=\frac{1}{n}\sum_{i=1}^n\left(h(\gamma X_i)-h(\gamma X_i\given \tilde\theta_i)\right)\nonumber\\
&=\frac{1}{n}\sum_{i=1}^n\left(h(\gamma X_i)-h(\gamma X_i-\tilde\theta_i\given \tilde\theta_i)\right)\nonumber\\
&\geq\frac{1}{n}\sum_{i=1}^n\left(h(\gamma X_i)-h(\gamma X_i-\tilde\theta_i)\right)\nonumber\\
&\geq\frac{1}{n}\sum_{i=1}^n\left(h(\gamma X_i)-h\left(\mathcal N(0,\mathbb E(\gamma X_i-\tilde\theta_i)^2\right)\right)\label{eqn:lem2.5}\\
&=\frac{1}{n}\sum_{i=1}^n\left(\frac{1}{2}\log\frac{c^4}{\sigma^2+c^2}-\frac{1}{2}\log\mathbb E(\gamma X_i-\tilde\theta_i)^2\right)\nonumber\\
&=\frac{1}{2}\log\frac{c^4}{\sigma^2+c^2}-\frac{1}{n}\sum_{i=1}^n\frac{1}{2}\log\mathbb E(\gamma X_i-\tilde\theta_i)^2\nonumber\\
&\geq\frac{1}{2}\log\frac{c^4}{\sigma^2+c^2}-\frac{1}{2}\log\mathbb E\,d(\tilde\theta^n,\gamma X^n)\label{eqn:lem2.8}\\
&=\frac{1}{2}\log\frac{\frac{c^4}{\sigma^2+c^2}}{\mathbb E\,d(\theta^n,\tilde\theta^n)-\frac{\sigma^2c^2}{\sigma^2+c^2}}\nonumber
\end{align}
where (\ref{eqn:lem2.5}) follows from the fact that the normal distribution maximizes the entropy for a given second moment, 
(\ref{eqn:lem2.8}) follows from the concavity of the $\log$ function,
and the other inequalities follow from the properties of
mutual information and entropy.
Since $\mathbb E\, d(\theta^n,\tilde\theta^n)\leq D$, we have 
\[
\frac{1}{n}I(X^n;\tilde\theta^n)\geq\frac{1}{2}\log\frac{\frac{c^4}{\sigma^2+c^2}}{D-\frac{\sigma^2c^2}{\sigma^2+c^2}}.
\]
On the other hand, a calculation shows that the following joint distribution
\begin{equation}
\tilde\theta^n\sim\mathcal N\left(0,\gamma^2(\sigma^2+c^2-D)I_n\right),\quad X^n\sim\mathcal N\left(\tilde\theta^n/\gamma,DI_n\right),\quad\theta^n\sim\mathcal N(\gamma X^n,\gamma\sigma^2I_n).\label{eqn:testdist}
\end{equation}
achieves the lower bound, which concludes the proof.
\end{proof} 

\begin{proof}[Proof of Theorem \ref{thm:lowerbound}]
Suppose that $M_n$ is an $(n,B)$-rate estimation code. 
Let $\pi_n$, the prior on $\theta^n$, be $\mathcal N(0,c^2I_n)$.
According to Lemmas \ref{lemmaratedistortion} and \ref{lem:normalcase}
\begin{align*}
\frac{\sigma^2c^2}{\sigma^2+c^2}+\frac{c^4}{\sigma^2+c^2}2^{-2B}&=\int \mathbb E_{X^n}d(\theta^n,\tilde\theta^n)d\pi_n(\theta^n)\\
&\leq \int \mathbb E_{X^n}d(\theta^n,M_n(X^n))d\pi_n(\theta^n),
\end{align*}
where $\tilde\theta^n$ follows the distribution specified in
(\ref{eqn:testdist}). It then suffices to show that as $n\to\infty$
\[
\int \mathbb E_{X^n}d(\theta^n,M_n(X^n))d\pi_n(\theta^n)\leq \sup_{\theta^n\in\Theta_n(c)}\mathbb E_{X^n}d(\theta^n,M_n(X^n)).
\]
In fact, if the above inequality holds, taking a supremum over $M_n\in\mathcal M_{n,B}$ gives the desired lower bound.
Recall that
$\Theta_n(c)=\{\theta^n:\frac{1}{n}\sum_{i=1}^n\theta_i^2\leq c^2\}$. 
Paralleling the argument in \cite{nussbaum1999minimax,wasserman:2006}, we have 
\begin{align*}
&\int \mathbb E_{X^n}d(\theta^n,M_n(X^n))d\pi_n(\theta^n) \\
&=\int_{\Theta_n(c)} \mathbb E_{X^n}d(\theta^n,M_n(X^n))d\pi_n(\theta^n) +\int_{\overline{\Theta_n(c)}} \mathbb E_{X^n}d(\theta^n,M_n(X^n))d\pi_n(\theta^n)\\
&\leq \sup_{\Theta_n(c)} \mathbb E_{X^n}d(\theta^n,M_n(X^n))+\int_{\overline{\Theta_n(c)}} \mathbb E_{X^n}d(\theta^n,M_n(X^n))d\pi_n(\theta^n).
\end{align*}
It remains to show that
\[
\int_{\overline{\Theta_n(c)}} \mathbb E_{X^n}d(\theta^n,M_n(X^n))d\pi_n(\theta^n)\longrightarrow 0.
\]
where $\overline{\Theta_n(c)}$ denotes the complement of $\Theta_n(c)$.
For a fixed $\delta\in(0,1)$, let $\pi_{n,\delta}$ be a $\mathcal N(0,c^2\delta^2I_n)$ prior on $\theta^n$. 
Replacing $\pi_n$ by $\pi_{n,\delta}$, and using the Cauchy-Schwarz inequality, we get 
\begin{align}
&\int_{\overline{\Theta_n(c)}} \mathbb E_{X^n}d(\theta^n,M_n(X^n))d\pi_{n,\delta}(\theta^n)\nonumber\\
&\leq 2\int_{\overline{\Theta_n(c)}}\frac{1}{n}\|\theta^n\|^2d\pi_{n,\delta}(\theta^n)+2\int_{\overline{\Theta_n(c)}}\mathbb E_{X^n}\frac{1}{n}\|M(X^n)\|^2d\pi_{n,\delta}(\theta^n)\nonumber\\
&\leq \frac{2}{n}\sqrt{\pi_{n,\delta}\left(\overline{\Theta_n(c)}\right)}\sqrt{\mathbb E_{\pi_{n,\delta}}\|\theta^n\|^4}+2c^2\pi_{n,\delta}\left(\overline{\Theta_n(c)}\right).\label{eqn:cauchyschwarz}
\end{align}
Now we bound the two terms in the formula above. First,
\begin{align*}
\pi_{n,\delta}\left(\overline{\Theta_n(c)}\right)&=\mathbb P\left(\frac{1}{n}\sum_{i=1}^n\theta_i^2>c^2\right)\\
&=\mathbb P\left(\frac{1}{n}\sum_{i=1}^n\left(\left(\frac{\theta_i}{\delta c}\right)^2-1\right)>\frac{1-\delta^2}{\delta^2}\right)\\
&\leq 2\exp\left(-\frac{n(1-\delta^2)^2}{8\delta^4}\right)
\end{align*}
where the last inequality is due to the following large deviation inequality: if $Z_1,\dots,Z_n\sim\mathcal N(0,1)$ and $0<t<1$, then
\[
\mathbb P\left(\left|\frac{1}{n}\sum_{i=1}^n(Z_i^2-1)\right|>t\right)\leq 2e^{-nt^2/8}.
\]
Next, we note that
\begin{align*}
\mathbb E_{\pi_{n,\delta}}\|\theta^n\|^4&=\sum_{i=1}^n\mathbb E_{\pi_{n,\delta}}\theta_i^4+\sum_{i\neq j}\mathbb E_{\pi_{n,\delta}}\theta_i^2\cdot\mathbb E_{\pi_{n,\delta}}\theta_j^2\\
&=n\mathbb E_{\pi_{n,\delta}}\theta_1^4+{n \choose 2}c^2\delta^2\\
&=O(n^2). 
\end{align*}
Therefore, we have from (\ref{eqn:cauchyschwarz})
\begin{align*}
&\int_{\overline{\Theta_n(c)}} \mathbb E_{X^n}d(\theta^n,M_n(X^n))d\pi_{n,\delta}(\theta^n)\\
&\leq \frac{2}{n}\cdot \sqrt{2}\exp\left(-\frac{n(1-\delta^2)^2}{16\delta^4}\right)O(n)+2c^2\exp\left(-\frac{n(1-\delta^2)^2}{8\delta^4}\right)\longrightarrow
0
\end{align*}
for any $\delta\in(0,1)$. The conclusion then follows by letting $\delta\uparrow1$.
\end{proof}

\def\P{{\mathbb P}}
\allowdisplaybreaks

\subsection{Proof of Theorem \ref{thm:codingmethod}}
\begin{proof}[Proof of Theorem \ref{thm:codingmethod}]
Suppose that $\|\theta^n\|/n=b^2\leq c^2$ and that $X_i\sim\mathcal N(\theta_i,\sigma^2)$. Writing $\hat b^2=\|X^n\|^2/n-\sigma^2$ and $\hat\gamma=\frac{\hat b^2}{\sigma^2+\hat b^2}$, we have the decomposition of the loss
\begin{align*}
d(\theta^n,\check\theta^n)&=\frac{1}{n}\left\|\check\theta^n-\theta^n\right\|^2\\
&=\frac{1}{n}\left\|\check\theta^n-\hat\gamma X^n+\hat\gamma X^n-\theta^n\right\|^2\\
&=\underbrace{\frac{1}{n}\left\|\check\theta^n-\hat\gamma X^n\right\|^2}_{A_1}+\underbrace{\frac{1}{n}\left\|\hat\gamma X^n-\theta^n\right\|^2}_{A_2}+\underbrace{\frac{2}{n}\langle\check\theta^n-\hat \gamma X^n,\hat\gamma X^n-\theta^n\rangle}_{A_3}.
\end{align*}
\begin{enumerate}[(i)]
\item Term $A_1$ characterizes the quantization error. It has the following decomposition
\begin{align*}
A_1=\frac{1}{n}\|\check\theta^n\|^2+\frac{1}{n}\hat\gamma^2\|X^n\|^2-\frac{2}{n}\langle \check\theta^n,\hat\gamma X^n \rangle.
\end{align*}
By Lemma \ref{lem:concentration} and Lemma \ref{lem:extremeangle} below, we have
\begin{align*}
\frac{1}{n}\|X^n\|^2-b^2-\sigma^2=O_P\left(\frac{1}{\sqrt{n}}\right),\quad\frac{\langle X^n,\check X^n\rangle}{\|X\|}-\sqrt{1-2^{-2B}}=O_P\left(\frac{\log n}{n}\right),
\end{align*}
and therefore
\begin{align*}
\frac{1}{n}\|\check\theta^n\|^2&=\frac{\check b^4(1-2^{-2B})}{\sigma^2+\check b^2}=\frac{b^4}{\sigma^2+b^2}(1-2^{-2B})+O_P\left(\frac{1}{\sqrt{n}}\right),\\
\frac{1}{n}\hat\gamma^2\|X^n\|^2&=\frac{\hat b^4}{\sigma^2+\hat b^2}=\frac{b^4}{\sigma^2+b^2}+O_P\left(\frac{1}{\sqrt{n}}\right),\\
\frac{2}{n}\langle \check\theta^n,\hat\gamma X^n \rangle&=\frac{2}{n}\frac{\hat b^2}{\sigma^2+\hat b^2} \sqrt{\frac{n\check b^4(1-2^{-2B})}{\sigma^2+\check b^2}}\cdot\langle X^n, \check X^n\rangle\\
&=\frac{2\hat b^2\check b^2\sqrt{1-2^{-2B}}}{\sqrt{(\sigma^2+\hat b^2)(\sigma^2+\check b^2)}}
\frac{\langle X,\check X\rangle}{\|X\|}\\
&= \frac{2b^4}{\sigma^2+b^2}(1-2^{-2B})+O_P\left(\frac{1}{\sqrt{n}}\right).
\end{align*}
which, combined together, gives us
\begin{align*}
A_1&= \frac{b^4}{\sigma^2+b^2}(1-2^{-2B})+\frac{b^4}{\sigma^2+b^2}-\frac{2b^4}{\sigma^2+b^2}(1-2^{-2B})+O_P\left(\frac{1}{\sqrt{n}}\right)\\
&=\frac{b^4}{\sigma^2+b^2}2^{-2B}+O_P\left(\frac{1}{\sqrt{n}}\right).
\end{align*}

\item Term $A_2$ does not involve random codebook, and is essentially the average loss of a James-Stein-type estimator. Suppose that $A_n$ is an $n\times n$ orthonormal matrix such that $A_n\theta^n=(\sqrt{n}b,0,\dots,0)^T$, which we will denote by $\tau^n$. Let $Y^n=A_nX^n$. Then $Y^n\sim\mathcal N(\tau^n,\sigma^2I)$ and $\|X^n\|=\|Y^n\|$. Expressing $A_2$ in terms of $Y^n$, we have
\begin{align*}
A_2&=\frac{1}{n}\left\|\hat \gamma X^n-\theta^n\right\|^2\\
&=\frac{1}{n}\|\hat\gamma A_nX^n-A_n\theta^n\|^2\\
&=\frac{1}{n}\|\hat\gamma Y^n-\tau^n\|^2\\
&=\frac{1}{n}\hat\gamma^2\|Y^n\|^2-2\hat\gamma b\frac{Y_1}{\sqrt{n}}+b^2\\
&=\frac{b^2\sigma^2}{\sigma^2+b^2}+O_P\left(\frac{1}{\sqrt{n}}\right).
\end{align*}
The last equality is because
\begin{align*}
\frac{1}{n}\|Y^n\|^2-b^2-\sigma^2=O_P\left(\frac{1}{\sqrt{n}}\right),\quad\frac{Y_1}{\sqrt{n}}-b =O_P\left(\frac{1}{\sqrt{n}}\right).
\end{align*}

\item Finally, it can be shown that the cross term satisfies
  $A_3=O_P(\sqrt{\frac{\log n}{n}})$ by exploiting the geometry of the
  vectors, and using the fact that most vectors are nearly orthogonal
  to each other in a high dimensional space.  In fact, write
\[
\hat\theta^n=\sqrt{\frac{n\hat b^4(1-2^{-2B})}{\hat b^2+\sigma^2}}\cdot \check X^n.
\]
Then in the decomposition
\[
\frac{2}{n}\langle\check\theta^n-\hat \gamma X^n,\hat\gamma X^n-\theta^n\rangle=\frac{2}{n}\langle\check\theta^n-\hat\theta^n,\hat\gamma X^n-\theta^n\rangle + \frac{2}{n}\langle \hat\theta^n,\hat\gamma X^n-\theta^n \rangle - \frac{2}{n}\langle \hat\gamma X^n,\hat\gamma X^n-\theta^n\rangle,
\]
the first term is $O_P(\frac{1}{\sqrt{n}})$, since
$\frac{1}{\sqrt{n}}\|\check\theta^n-\hat\theta^n\|=O_P(\frac{1}{\sqrt{n}})$
and $\frac{1}{\sqrt{n}}\|\hat\gamma X^n-\theta^n\|$ has bounded second
moment, and the third term is
\begin{align*}
\frac{2}{n}\langle \hat\gamma X^n,\hat\gamma X^n-\theta^n\rangle&=\frac{2}{n}\langle \hat\gamma Y^n, \hat\gamma Y^n-\tau^n\rangle\\
&=\frac{2}{n}\hat\gamma^2\|Y^n\|^2-2\hat\gamma b\frac{Y_1}{\sqrt{n}}=O_P\left(\frac{1}{\sqrt{n}}\right).
\end{align*}
Now consider the second term
\begin{align*}
\frac{2}{n}\langle\hat\theta^n,\hat\gamma X^n-\theta^n\rangle=\frac{2}{n}\langle \hat\theta^n, \tilde\gamma X^n-\theta^n \rangle+\frac{2}{n}(\hat\gamma-\tilde\gamma)\langle \hat\theta^n, X^n \rangle
\end{align*}
where
\[
\tilde\gamma\triangleq\frac{\sum_{i=1}^n\theta_iX_i}{\|X^n\|^2},\quad\text{satisfying }\frac{2}{n}(\hat\gamma-\tilde\gamma)\langle \hat\theta^n, X^n \rangle=O_P\left(\frac{1}{\sqrt{n}}\right)\text{ and }\langle X^n,\tilde\gamma X^n-\theta^n \rangle=0.
\]
Thus, we are left with one last term to analyze:
\begin{align*}
\frac{2}{n}\langle \hat\theta^n, \tilde\gamma X^n-\theta^n \rangle&=\frac{2}{n}\sqrt{\frac{n\hat b^4(1-2^{-2B})}{\hat b^2+\sigma^2}}\langle \check X^n,\tilde\gamma X^n-\theta^n\rangle\\
&=\sqrt{\frac{4\hat b^4(1-2^{-2B})}{\hat b^2+\sigma^2}}\langle \check X^n,\frac{1}{\sqrt{n}}(\tilde\gamma X^n-\theta^n)\rangle
\end{align*}
The scaling factor in front of the inner product is some constant plus
an $O_P(\frac{1}{\sqrt{n}})$ term, so we consider the inner
product. Notice that the projection of $\check X^n$ onto the
orthogonal space of $X^n$, $\text{Proj}_{{X^n}^\perp}(\check X^n)$, is
independent of $X^n$. Furthermore, by symmetry,
$\text{Proj}_{{X^n}^\perp}(\check X^n)$ has a spherical distribution
in $\mathbb R^{n-1}$, and has a length
$\sqrt{1-2^{-2B}}+O_P(\frac{\log n}{n})$. That is, we can write
\[
\text{Proj}_{{X^n}^\perp}(\check X^n)=L_n\cdot U^n
\]
where $U^n$ follows the uniform distribution on sphere $\mathbb S^{n-1}$ and $L_n=\sqrt{1-2^{-2B}}+O_P(\frac{\log n}{n})$. Conditioning on $X^n$, since $\langle X^n,\tilde\gamma X^n-\theta^n\rangle=0$, we have
\begin{align*}
\MoveEqLeft \mathbb P\left(\langle \check X^n,\frac{1}{\sqrt{n}}(\tilde\gamma
  X^n-\theta^n)\rangle > t \given X^n=x^n\right)\\
&=\mathbb
P\left(\langle \text{Proj}_{{X^n}^\perp}(\check
  X^n),\text{Proj}_{{X^n}^\perp}(\frac{1}{\sqrt{n}}(\tilde\gamma
  X^n-\theta^n))\rangle > t \given X^n=x^n\right)\\
&=\mathbb P\left(L_n\|\frac{1}{\sqrt{n}}(\tilde\gamma x^n-\theta^n)\|\langle U^n,e^n\rangle>t\right)\\
&\leq K_1\sqrt{n}\left(1-\frac{t^2}{K_2\|\frac{1}{\sqrt{n}}(\tilde\gamma x^n-\theta^n)\|^2}\right)^{\frac{n-2}{2}}.
\end{align*}
where $K_1$ and $K_2$ are positive constants, and the last inequality
follows from Lemma \ref{lem:orthog} below. It then follows that
\begin{align*}
\MoveEqLeft \mathbb P\left(\langle \check X^n,\frac{1}{\sqrt{n}}(\tilde\gamma
  X^n-\theta^n)\rangle > t\right)\\
&=\int \mathbb P\left(\langle \check
  X^n,\frac{1}{\sqrt{n}}(\tilde\gamma X^n-\theta^n)\rangle > t\given X^n=x^n\right)p_{X^n}(x^n)dx^n\\
&\leq K_1\sqrt{n}\int \left(1-\frac{t^2}{K_2\|\frac{1}{\sqrt{n}}(\tilde\gamma x^n-\theta^n)\|^2}\right)^{\frac{n-2}{2}}p_{X^n}(x^n)dx^n\\
&\leq K_1\sqrt{n}\left(1-\frac{t^2}{K_2'}\right)^{\frac{n-2}{2}}+\mathbb P\left(\frac{1}{n}\|X^n\|^2>b^2+\sigma^2+K_3\right).
\end{align*}
for positive constants $K_1$, $K_2$ and $K_3$. This implies that
$\langle \check X^n,\frac{1}{\sqrt{n}}(\tilde\gamma
X^n-\theta^n)\rangle=O_P(\sqrt{\frac{\log n}{n}})$ and thus $A_3 =O_P(\sqrt{\frac{\log n}{n}})$.
\end{enumerate}
Combining the above analyses for $A_1$, $A_2$ and $A_3$ together gives us the theorem. 
\end{proof}

\begin{lemma}\label{lem:concentration}
Suppose that $X_i\overset{\text{ind.}}{\sim} N(\theta_i,\sigma^2)$, for $i=1,\dots,n$ and that $\frac{1}{n}\sum_{i=1}^n\theta_i^2=b^2$. Then
\begin{align*}
\mathbb P\left(\left|\frac{1}{n}\sum_{i=1}^nX_i^2-b^2-\sigma^2\right|\geq t\right)\leq2\exp\left(-\frac{nt^2}{32\sigma^4}\right)+\frac{8\sigma b}{\sqrt{2\pi nt^2}}\exp\left(-\frac{nt^2}{32\sigma^2b^2}\right).
\end{align*}
Specifically, if we write $\hat b^2=\|X\|^2/n-\sigma^2$, we have $\hat b^2-b^2=O_P(\frac{1}{\sqrt{n}})$.
\end{lemma}
\begin{proof} Writing $X_i=\theta_i+\epsilon_i$, we have
\begin{align*}
\MoveEqLeft \mathbb
P\left(\left|\frac{1}{n}\sum_{i=1}^nX_i^2-b^2-\sigma^2\right|>t\right)\\
&=\mathbb P\left(\left|\frac{1}{n}\sum_{i=1}^n(\epsilon_i^2-\sigma^2)+\frac{2}{n}\sum_{i=1}^n\theta_i\epsilon_i\right|>t\right)\\
&\leq\mathbb P\left(\left|\frac{1}{n}\sum_{i=1}^n(\epsilon_i^2-\sigma^2)\right|>\frac{t}{2}\right)+\mathbb P\left(\left|\frac{2}{n}\sum_{i=1}^n\theta_i\epsilon_i\right|>\frac{t}{2}\right)\\
&\leq\mathbb P\left(\left|\frac{1}{n}\sum_{i=1}^n\left(\left(\frac{\epsilon_i}{\sigma}\right)^2-1\right)\right|>\frac{t}{2\sigma^2}\right)+\mathbb P\left(\left|\frac{2}{n}\mathcal N(0,n\sigma^2b^2)\right|>\frac{t}{2}\right)\\ 
&\leq2\exp\left(-\frac{nt^2}{32\sigma^4}\right)+\frac{8\sigma b}{\sqrt{2\pi nt^2}}\exp\left(-\frac{nt^2}{32\sigma^2b^2}\right)
\end{align*}
where the last inequality follows from the previously mentioned large
deviation inequality and the upper tail inequality for the normal distribution.
\end{proof}

\begin{lemma}[Lemma 4.1 from \cite{tony2012phase}]\label{lem:pearson}
Suppose that $Y$ is uniformly distributed on the $n$-dimensional unit sphere $\mathbb S^{n-1}$. For $x\in\mathbb R^n$ such that $\|x\|_2=1$, the inner product $\rho=\langle x,Y\rangle$ between $x$ and $Y$has density function
\[
f(\rho)=\frac{1}{\sqrt \pi}\frac{\Gamma(\frac{n}{2})}{\Gamma(\frac{n-1}{2})}(1-\rho^2)^{\frac{n-3}{2}}I(|\rho|<1).
\]
\end{lemma}
\begin{lemma}\label{lem:extremeangle}
Suppose that $p=e^{n\beta}$ and $Y_1,\dots,Y_p$ are independent and identically distributed with a uniform distribution on the $n$-dimensional sphere $\mathbb S^{n-1}$. For a fixed unit vector $x\in\mathbb R^n$, let $\rho_i=\langle x,Y_i\rangle$ and $L_n=\max_{1\leq i\leq p}\,\rho_i$ . Then $L_n\to \sqrt{1-e^{-2\beta}}$ in probability as $n\to\infty$. Furthermore, $L_n-\sqrt{1-e^{-2\beta}}=O_P(\frac{\log n}{n})$.
\end{lemma}
\begin{proof} Let $k_n=\frac{n}{\log n}$. For any fixed $u\in\mathbb R$
\begin{align*}
&\P\left(k_n\left(L_n-\sqrt{1-e^{-2\beta}}\right)\leq u\right)\\
&=\P\left(L_n\leq \frac{u}{k_n}+\sqrt{1-e^{-2\beta}}\right)\\
&=\P\left(\rho_1\leq\frac{u}{k_n}+\sqrt{1-e^{-2\beta}}\right)^p\\
&=\left(1-\int_{u/k_n+\sqrt{1-e^{-2\beta}}}^1\frac{1}{\sqrt{\pi}}\frac{\Gamma(\frac{n}{2})}{\Gamma(\frac{n-1}{2})}(1-\rho^2)^{\frac{n-3}{2}}d\rho\right)^p\\
&\sim\left(1-\frac{\sqrt{n}}{\sqrt{2\pi}(n-3)(\frac{u}{k_n}+\sqrt{1-e^{-2\beta}})}\left(1-\left(\frac{u}{k_n}+\sqrt{1-e^{-2\beta}}\right)^2\right)^{\frac{n-1}{2}}\right)^p\\
&\sim\exp\left(-p\cdot\frac{\sqrt{n}}{\sqrt{2\pi}(n-3)(\frac{u}{k_n}+\sqrt{1-e^{-2\beta}})}\left(1-\left(\frac{u}{k_n}+\sqrt{1-e^{-2\beta}}\right)^2\right)^{\frac{n-1}{2}} \right)\\
&\triangleq \exp (-M).
\end{align*}
Taking the logarithm of the exponent $M$, we get
\begin{align*}
& \log M =\log p+\frac{n-1}{2}\log\left(1-\left(\frac{u}{k_n}+\sqrt{1-e^{-2\beta}}\right)^2\right)+\log\frac{\sqrt{n}}{\sqrt{2\pi}(n-3)(\frac{u}{k_n}+\sqrt{1-e^{-2\beta}})}\\
&\sim n\beta +\frac{n-1}{2}\log\left(e^{-2\beta}-\frac{u^2}{k_n^2}-\frac{2u}{k_n}\sqrt{1-e^{-2\beta}}\right)-\frac{1}{2}\log n-\frac{1}{2}\log\left(1-e^{-2\beta}\right)-\frac{1}{2}\log(2\pi)\\
& \sim n\beta-(n-1)\beta-\frac{n-1}{2}\frac{u^2}{k_n^2}-(n-1)\frac{u}{k_n}\sqrt{1-e^{-2\beta}}-\frac{1}{2}\log n-\frac{1}{2}\log\left(1-e^{-2\beta}\right)-\frac{1}{2}\log(2\pi)\\
& \sim \beta-\left(u\sqrt{1-e^{-2\beta}}+\frac{1}{2}\right)\log n-\frac{1}{2}\log\left(1-e^{-2\beta}\right)-\frac{1}{2}\log(2\pi).
\end{align*}
If $u>0$, then as $n\to\infty$, $M\to0$, and thus $\P\left(k_n\left(L_n-\sqrt{1-e^{-2\beta}}\right)\leq u\right)\to0$. If $u<-\frac{1}{2\sqrt{1-2^{-2\beta}}}$, then as $n\to\infty$, $M\to\infty$, and hence $\P\left(k_n\left(L_n-\sqrt{1-e^{-2\beta}}\right)\leq u\right)\to1$. We can then conclude that $|L_n-\sqrt{1-e^{-2\beta}}|=O_P(\frac{\log n}{n})$.
\end{proof}

\begin{lemma}
\label{lem:orthog}
Let $U$ have a uniform distribution on the unit sphere $\mathbb
S^{n-1}$ and let $x\in\mathbb R^n$ be a fixed vector. Then
\[
\mathbb P\left(|\langle U,x\rangle|>\epsilon\right)\leq K\sqrt{n}(1-\epsilon^2)^{\frac{n-2}{2}}.
\]
for all $n\geq 2$ and $\epsilon\in (0,1)$, where $K$ is a universal constant. Therefore,
\[
\langle U,x\rangle=O_P\left(\sqrt{\frac{\log n}{n}}\right).
\]
\end{lemma}

\begin{proof}
This is a direct result from Proposition 1 in \cite{cai2013distributions}.
\end{proof}

\section*{Acknowledgements}
Research supported in part by NSF grants IIS-1116730, 
AFOSR grant FA9550-09-1-0373, ONR grant
N000141210762, and an Amazon AWS in Education Machine Learning
Research grant.   

The authors thank 
Andrew Barron, John Duchi, and Alfred Hero for valuable comments on
this work.


\begin{thebibliography}{10}

\bibitem{cai2013distributions}
T.~Tony Cai, Jianqing Fan, and Tiefeng Jiang.
\newblock Distributions of angles in random packing on spheres.
\newblock {\em The Journal of Machine Learning Research}, 14(1):1837--1864,
  2013.

\bibitem{tony2012phase}
T.~Tony Cai and Tiefeng Jiang.
\newblock Phase transition in limiting distributions of coherence of
  high-dimensional random matrices.
\newblock {\em Journal of Multivariate Analysis}, 107:24--39, 2012.

\bibitem{chandrasekaran:13}
Venkat Chandrasekarana and Michael~I. Jordan.
\newblock Computational and statistical tradeoffs via convex relaxation.
\newblock {\em PNAS}, 110(13):E1181--E1190, March 2013.

\bibitem{donoho2000wald}
David~L. Donoho.
\newblock Wald lecture {I}: {C}ounting bits with {K}olmogorov and {S}hannon.
\newblock 2000.

\bibitem{draper2004side}
Stark~C. Draper and Gregory~W. Wornell.
\newblock Side information aware coding strategies for sensor networks.
\newblock {\em Selected Areas in Communications, IEEE Journal on},
  22(6):966--976, 2004.

\bibitem{2041-8205-713-2-L87}
Jon M.~Jenkins et~al.
\newblock Overview of the {K}epler science processing pipeline.
\newblock {\em The Astrophysical Journal Letters}, 713(2):L87, 2010.

\bibitem{gallager:1968}
Robert~G. Gallager.
\newblock {\em Information Theory and Reliable Communication}.
\newblock John Wiley \& Sons, 1968.

\bibitem{johnstone2002function}
Iain~M. Johnstone.
\newblock Function estimation and {G}aussian sequence models.
\newblock 2002.
\newblock Unpublished manuscript.

\bibitem{nussbaum1999minimax}
Michael Nussbaum.
\newblock Minimax risk: {P}insker bound.
\newblock {\em Encyclopedia of Statistical Sciences}, 3:451--460, 1999.

\bibitem{pinsker1980optimal}
Mark~Semenovich Pinsker.
\newblock Optimal filtering of square-integrable signals in {G}aussian noise.
\newblock {\em Problemy Peredachi Informatsii}, 16(2):52--68, 1980.

\bibitem{tsybakov:2008}
Alexandre~B. Tsybakov.
\newblock {\em Introduction to Nonparametric Estimation}.
\newblock Springer Series in Statistics, 1st edition, 2008.

\bibitem{venkataramanan2013lossy}
Ramji Venkataramanan, Tuhin Sarkar, and Sekhar Tatikonda.
\newblock Lossy compression via sparse linear regression: {C}omputationally
  efficient encoding and decoding.
\newblock In {\em IEEE International Symposium on Information Theory (ISIT)},
  pages 1182--1186. IEEE, 2013.

\bibitem{wasserman:2006}
Larry Wasserman.
\newblock {\em All of Nonparametric Statistics}.
\newblock Springer-Verlag, 2006.

\bibitem{Wong:Shen:1995}
Wing~Hung Wong and Xiaotong Shen.
\newblock Probability inequalities for likelihood ratios and convergence rates
  of sieve {MLE}s.
\newblock {\em The Annals of Statistics}, 23:339--362, 1995.

\bibitem{Yang:Barron:1999}
Yuhong Yang and Andrew Barron.
\newblock Information-theoretic determination of minimax rates of convergence.
\newblock {\em The Annals of Statistics}, 27(5):1564--1599, 1999.

\bibitem{zhang2013information}
Yuchen Zhang, John Duchi, Michael Jordan, and Martin~J. Wainwright.
\newblock Information-theoretic lower bounds for distributed statistical
  estimation with communication constraints.
\newblock In {\em Advances in Neural Information Processing Systems}, pages
  2328--2336, 2013.

\end{thebibliography}

\end{document}